\documentclass[11pt]{amsart}
\usepackage{amssymb}
\usepackage{amsmath}
\usepackage{}
\usepackage{amssymb,latexsym}
\usepackage{enumerate}

\newcommand{\op} {\overline{\partial}}

\newcommand{\C}{\ensuremath{\mathbb{C}}}

\makeatletter
\newcommand{\sumprime}{\if@display\sideset{}{'}\sum%
            \else\sum'\fi}
\makeatother

\begin{document}

\numberwithin{equation}{section}

% define theorem environments
\newtheorem{theorem}{Theorem}[section]
\newtheorem{proposition}[theorem]{Proposition}
\newtheorem{conjecture}[theorem]{Conjecture}
\def\theconjecture{\unskip}
\newtheorem{corollary}[theorem]{Corollary}
\newtheorem{lemma}[theorem]{Lemma}
\newtheorem{observation}[theorem]{Observation}
\newtheorem{definition}{Definition}
\numberwithin{definition}{section} %\def\thedefinition{\unskip}
\newtheorem{remark}{Remark}
\def\theremark{\unskip}
\newtheorem{question}{Question}
\def\thequestion{\unskip}
\newtheorem{example}{Example}
\def\theexample{\unskip}
\newtheorem{problem}{Problem}

\thanks{Research supported by the Key Program of NSFC No. 11031008.}

\address{Department of Mathematics, Tongji University, Shanghai, 200092, China}
\email{1113xuwang@tongji.edu.cn}

\title[Variation of the Bergman kernels]{Variation of the Bergman kernels under deformations of complex structures}
 \author{Xu Wang}
\date{}
\maketitle

\begin{abstract}  Inspired by Berndtsson's work on the subharmonicity property of the Bergman kernel, we give a local variation formula of the full Bergman kernels associated to deformations of complex manifolds. In compact case, it follows from the reproducing property of the Bergman kernel and the curvature formula of the $0$-th direct image sheaf. In general, following Schumacher's idea, we use the Lie derivative to compute the variation. An equivalent criterion for the triviality of holomorphic motions of planar domains in terms of the Bergman kernel is given as an application.
\bigskip


\smallskip

\end{abstract}

\section{Basic definitions and results}

Let $D$ be a pseudoconvex domain in $\C^k_t \times \C^n_z$ and let $\phi$ be a plurisubharmonic function in $D$. Denote by $U$ the projection of $D$ to $\C^k_t $. Put
$$
D_t=\{z\in \C^n:~(t,z)\in D\}, \  \phi^t=\phi|_{D_t}, \ \forall\ t\in U.
$$
Denote by $A^2(D_t,e^{-\phi^t})$ the Bergman space of weighted $L^2$ holomorphic functions in $D_t$. Denote by $K^t(z,\bar w)$ the Bergman kernel of $A^2(D_t,e^{-\phi^t})$. Our start point is the following rather remarkable result of Berndtsson (Theorem 1.1 in \cite{Bern06}):

\begin{theorem}\label{th:Bern1} The function $\log K^t(z,\bar z)$ is plurisubharmonic, or identically equal to $-\infty$ in $D$.
\end{theorem}

The most important ingredient in the proof of Theorem~\ref{th:Bern1} is a particular case of a result from \cite{Bern09}:

\begin{theorem}\label{th:Bern2} Assume that $D=D_0\times U$, where $D_0$ is a smoothly bounded strictly pseudoconvex domain and $\phi$ is smooth up to the boundary. Then the curvature of $A^2(D_0,e^{-\phi^t})\times U$ is nonnegative.
\end{theorem}

For applications of the above two theorems, see \cite{Bern06}, \cite{Bern09}, \cite{Bern09a}, \cite{Bern13}, \cite{BernPaun08}. See also \cite{MY04} for the first results in this direction.

Let $f$ be a complex valued function on $D_0\times U$ such that $f^t$ lies in $L^2(D_0)$ and $f$ depends smoothly on $t$. Then
\begin{equation}\label{eq:htf}
    \tau(t): \ h^t\mapsto\int_{D_0}h^t\bar {f^t}e^{-\phi^t}
\end{equation}
defines a holomorphic section of the dual of $A^2(D_0,e^{-\phi^t})\times U$ if and only if
\begin{equation}\label{eq:df}
    D'f^t:=e^{\phi^t}\partial/\partial t(f^te^{-\phi^t})\ \bot \ A^2(D_0,e^{-\phi^t}),\ \forall \ t\in U.
\end{equation}
Notice that
\begin{equation}\label{eq:kf}
    K_f(t):=\int_{D_0\times D_0}K^t(z,\bar w)f^t(w)\overline{f^t(z)}e^{-\phi^t(z)-\phi^t(w)}=||\tau(t)||^2.
\end{equation}
Thus Theorem~\ref{th:Bern2} implies that if $f$ satisfies \eqref{eq:df} then $\log K_f(t)$ is plurisubharmonic.

For general $D$ and $\phi$, let $f$ be a complex valued function on $D$ such that $f^t\in L^2(D_t,e^{\phi^t})$ for every $t\in U$. Assume that $f$ dose not depend on $t$, i.e. $f(z,t_1)=f(z,t_2)$ as long as $(z,t_1), (z,t_2)\in D$. We also put \begin{equation*}
    K_{fe^\phi}(t)=\int_{D_t\times D_t}K^t(z,\bar w)f(w)\overline{f(z)}.
\end{equation*}
By using a similar argument as in the proof of Theorem~\ref{th:Bern1}, we get another form of Berndtsson's Theorem:

\begin{theorem}\label{th:Bern} Let $u^t$ be the minimal solution of $\op u=\op(fe^{\phi^t})$ in $L^2(D_t,e^{-\phi^t})$. Then the function $\log(||fe^{\phi^t}||^2-||u^t||^2)=\log K_{fe^{\phi}}(t)$ is plurisubharmonic, or identically equal to $-\infty$ in $U$.
\end{theorem}

In particular, if $f$ is invariant under rotations around $z$ then $\log K_f(t)=\log K^t(z,\bar z)+C$, where $C$ is a constant that depends only on $f$. By the Oka trick of variation of the domain, Theorem~\ref{th:Bern1} follows from Theorem~\ref{th:Bern}.

Theorem~\ref{th:Bern} suggests to study variation of the full Bergman kernels $K^t(z,\bar w)$, not only the Bergman kernels on the diagonal. We shall study variation of the full Bergman kernels under deformations of complex structures.

Let $p:\mathcal X\rightarrow\mathbb D$ be a holomorphic submersion from a complex manifold $\mathcal X$ onto the unit disc $\mathbb D$. Assume that all fibres $X_t:=p^{-1}(t)$ are connected. Let $\mathcal {L}$ be a holomorphic line bundle over $\mathcal X$ equipped with a smooth metric $e^{-\phi}$. Denote by $L_t$ the restriction of $\mathcal L$ to $X_t$.  Take a locally finite covering of $\mathcal X$ by sufficiently small coordinate neighborhoods
$$
\{(t,\zeta_\alpha)=(t,\zeta_\alpha^1,\cdots,\zeta_\alpha^n): U_\alpha\rightarrow\mathbb C^{n+1}\}
$$
such that $\{\zeta_\alpha: U_\alpha\cap X_t\neq\emptyset\}$ gives the complex structure of $X_t$ and each peace $\mathcal L|_{U_\alpha}$ has a product structure $\mathcal L|_{U_\alpha}=\mathbb C\times U_\alpha$. Then
$$
e_\alpha: (t,\zeta_\alpha)\mapsto(1,t,\zeta_\alpha)
$$
defines a holomorphic section of $\mathcal L$ over $U_\alpha$. By Kodaira-Spencer's definition (see Page 46 in \cite{KSb}), the associate fibre coordinate of the local section $e_\alpha$ is the admissible fibre coordinate of $\{L_t: t\in \mathbb D\}$ on $U_\alpha$. Let $K_t$ be the canonical line bundle of $X_t$. Denote by $E$ the holomorphic line bundle over $\mathcal X$ defined by $\{d\zeta_\alpha \otimes e_\alpha \}_{\alpha}$, where $d\zeta_\alpha:=d\zeta_\alpha^1\wedge\cdots\wedge d\zeta_\alpha^n$. Then we have $E_t:=E|_{X_t}=K_t+L_t$ and the associated fibre coordinate of the local section $d\zeta_\alpha\otimes e_\alpha$ is the admissible fibre coordinate of $\{E_t: t\in \mathbb D\}$ on $U_\alpha$.

Denote by $\mathcal H(X_t,E_t)$ the space of $L^2$ holomorphic n-forms on $X_t$ with values in $L_t$. The Bergman kernel $K^t$ of $\mathcal H(X_t,E_t)$ is the integral kernel of the orthogonal projection from the space of smooth $(n,0)$-forms with values in $L_t$ onto $\mathcal H(X_t,E_t)$. It may be represented by the holomorphic section
\begin{equation}\label{eq:bk}
    \sum_j u_j(x)\otimes \overline{u_j(y)}
\end{equation}
of the pull back line bundle $E_t\boxtimes\overline {E_t}$ over $X_t\times \overline{X_t}$, where $\{u_j\}$ is a complete orthonormal base of $\mathcal H(X_t,E_t)$. Locally, one may write
\begin{equation}\label{eq:bklocal}
    K^t=K^t(\zeta_\alpha,\overline{\eta_\beta})~ d\zeta_\alpha\otimes e_\alpha \otimes \overline{d\eta_\beta\otimes e_\beta}.
\end{equation}
Then $K^t(\zeta_\alpha,\overline{\eta_\beta})$ is the admissible fibre coordinate of $K^t$. We say that $K^t$ depends smoothly on $t$ if $K^t(\zeta_\alpha,\overline{\eta_\beta})$ depends smoothly on $t$.

\medskip
\emph{Throughout this paper (unless otherwise stated), we shall denote by $(t,\zeta_\alpha), (t,\eta_\beta), (t,\mu_\gamma)$ (resp. $z_\alpha,w_\beta$) the local coordinates of $\mathcal X$ (resp. $X$).}
\medskip

The following formula
\begin{equation}\label{eq:reproducing1}
    K^t(\zeta_\alpha,\overline{\eta_\beta})=\int_{X_t}K^t(\mu_\gamma,\overline{\eta_\beta})\overline{K^t(\mu_\gamma,\overline{\zeta_\alpha})}
    e^{-\phi^t(\mu_\gamma)}i^{n^2}d\mu_\gamma\wedge \overline{d\mu_\gamma}
\end{equation}
shall play a central role in this paper. Put
\begin{equation}\label{eq:K1}
    K^{t, \overline{\zeta_\alpha}}=K^t(\mu_\gamma,\overline{\zeta_\alpha}) \ d\mu_\gamma\otimes e_\gamma\in \mathcal H (X_t, E_t), \ \forall\ (t,\zeta_\alpha)\in U_\alpha.
\end{equation}
Omit $\alpha, \beta$, we then have
\begin{equation}\label{eq:K2}
    K^t(\zeta,\bar{\eta})=i^{n^2}\int_{X_t} \{K^{t, \bar{\eta}}, K^{t, \bar{\zeta}}\}=\langle\langle K^{t, \bar{\eta}}, K^{t, \bar{\zeta}}\rangle\rangle_t,
\end{equation}
where $\{\cdot,\cdot\}$ is the canonical sesquilinear pairing. Now the reproducing formula may be written as
\begin{equation}\label{eq:reproducing2}
    u^t(\zeta)=\langle\langle u^t, K^{t, \bar{\zeta}}\rangle\rangle_t,\ \forall\ u^t\in \mathcal H(X_t,E_t).
\end{equation}
Put
$$
u_t=\frac{\partial}{\partial t} u, \ \ u_{\bar t}=\frac{\partial}{\partial \bar t} u, \ \ u_{t\bar t}=\frac{\partial^2}{\partial t\partial \bar t} u.
$$
If $||K_{\bar t}^{t,\bar\zeta}||_t<\infty$, by the reproducing formula, we get the first order local variation formula:
\begin{equation}\label{eq:firstorder}
   K_t^{t}(\zeta,\bar\eta)=\langle\langle K^{t,\bar\eta}, K_{\bar t}^{t, \bar{\zeta}}\rangle\rangle_t.
\end{equation}
If $\mathcal H:=\{\mathcal H(X_t, E_t): t\in\mathbb D\}$ is a well defined holomorphic vector bundle over $\mathbb D$, then \eqref{eq:firstorder} implies that $D_tK^{t,\bar\eta}=0$, where $D_t dt$ is the $(1,0)$-component of the Chern connection on $\mathcal H$. Thus
\begin{equation}\label{eq:secondorder}
    K_{t\bar t}^{t}(\zeta,\bar\eta)=\langle\langle K_{\bar t}^{t,\bar\eta}, K_{\bar t}^{t, \bar{\zeta}}\rangle\rangle_t+ \langle\langle K^{t,\bar\eta}, \Theta K^{t, \bar{\zeta}} \rangle\rangle_t
    =\langle\langle K_{\bar t}^{t,\bar\eta}, K_{\bar t}^{t, \bar{\zeta}}\rangle\rangle_t+\Theta K^{t, \bar{\eta}}(\zeta),
\end{equation}
where $\Theta:=[D_t,\partial/\partial \bar t~ ]$ is the curvature of $\mathcal H$. Hence the second order local variation formula follows from the curvature formula of $\mathcal H$.

In general, as explained in \cite{Bern09}, $\mathcal H$ is not locally trivial. Inspired by \cite{Maitani84} and \cite{Sch12}, we shall use the Lie derivative to compute the variation.

\begin{definition}\label{de:vector} Let $\Psi:M\rightarrow N$ be a smooth submersion of differential manifolds. Let $V$ be a smooth real vector field on $M$. Assume that $\Psi_*(V(y))\equiv V_N(x),  \ \forall\ y\in \Psi^{-1}(x),\  \forall \ x \in N$. $V$ is said to be $\Psi$-admissible on $\Psi^{-1}(x_0)$ if there exists a diffeomorphism $\Phi: \gamma\times \Psi^{-1}(x_0)\rightarrow \Psi^{-1}(\gamma)$ such that $\Psi\circ\Phi$ is the canonical projection and $\Phi_*(V_N)=V$ on $\Psi^{-1}(\gamma)$, where $\gamma$ is an integral curve of $V_N$ passing through $x_0$.
\end{definition}

Let $V$ be a $p$-admissible smooth $(1,0)$-vector field on $\mathcal X$ (i.e. both ${\rm Re}V$ and ${\rm Im}V$ are $p$-admissible) such that $p_*V=\partial/\partial t$. Denote by $C^{\infty}_{\bullet,\bullet}(X_t,L_t)$ the graded algebra of smooth forms on $X_t$ with values in $L_t$. Let $u^t\in \mathcal C^{\infty}_{\bullet,\bullet}(X_t,L_t)$ whose admissible fibre coordinates depend smoothly on $t$. Let $i_t:X_t\hookrightarrow\mathcal X$ be the inclusion mapping. Then
\begin{equation}\label{eq:Liederivative}
    \mathcal L^t_{V,\phi}u^t:=i_t^*\left(e^{\phi}\mathcal L_{V}(e^{-\phi}u^t)\right),  \ \ \mathcal L^t_{\overline{V}}u^t:=i_t^*\left(\mathcal L_{\overline{V}}u^t\right),
\end{equation}
are globally well defined on $X_t$, where $\mathcal L_{V}, \mathcal L_{\overline{V}}$ are the usual Lie derivatives. What's more, if $u^t$ are holomorphic $p$-forms with values in $L_t$, then $\mathcal L^t_{\overline{V}}u^t=u_{\bar t}^t$. If $V$ is integrable, i.e. $[V,\overline{V}]=0$, then $\phi_{V\overline{V}}:=V\overline{V}\phi=\overline{V}V \phi$ is also globally well defined.

The complex Lie derivative $\mathcal L_V^\C$ introduced by Berndtsson (see Page 466 in \cite{Bern09a}) is defined as follows:
\begin{equation*}
    \mathcal L_V^\C:=\partial\delta_V+\delta_V \partial,
\end{equation*}
where $\delta_V$ means contraction of a form with a vector field. Put $\op^t=\op|_{X_t}, \ \partial^t=\partial|_{X_t}$. By Cartan's formula, we have
\begin{equation}\label{eq:relation1}
    \mathcal L^t_{V,\phi}u^t=\mathcal L^{t,\C}_{V,\phi}u^t+\delta_{\op^t V}u^t,
\end{equation}
where $\mathcal L^{t,\C}_{V,\phi}u^t:=i_t^*\left(e^{\phi}\mathcal L^\C_{V}(e^{-\phi}u^t)\right)$. Notice that $\op^t V$ is a representative of the Kodaira-Spencer class of $X_t\hookrightarrow \mathcal X$.

In order to compute the variation of the Bergman kernels, we have to assume that $K^t$ is sufficiently regular. Put
$$
A_0^t=\{K^{t,\bar\eta},K^{t,\bar\zeta}\}, \ \ A_1^t=\{\mathcal L^t_{V,\phi}K^{t,\bar\eta},K^{t,\bar\zeta}\}, \ \ A_2^t=\{K_{\bar t}^{t,\bar\eta},K^{t,\bar\zeta}\},
$$
we need two assumptions: first order conditions
\begin{equation}\label{eq:main1}
       \frac{\partial}{\partial t}\int_{X_t}A_0^t=\int_{X_t}\mathcal L^t_V A_0^t, \ \
        ||\mathcal L^t_{V,\phi}K^{t,\bar\zeta}||_t<\infty, \ ||K_{\bar t}^{t,\bar\zeta}||_t<\infty,
\end{equation}
and second order conditions
\begin{equation}\label{eq:main2}
    \frac{\partial}{\partial t}\int_{X_t}A_j^t=\int_{X_t}\mathcal L^t_V A_j^t, \ j=1,2, \ \
    ||\mathcal L^t_{V,\phi}K_{\bar t}^{t,\bar\zeta}||_t<\infty, \ ||\mathcal L^t_{\overline{V}}\mathcal L^t_{V,\phi}K^{t,\bar\zeta}||_t<\infty.
\end{equation}
Our main theorem (generalization of (2.3) in \cite{Bern06}) can be stated as follows:

\begin{theorem}[Main Theorem]\label{th:Main} Let $V$ be a $p$-admissible smooth $(1,0)$-vector field on $\mathcal X$ such that $p_*V=\partial/\partial t$. If \eqref{eq:main1} is satisfied, then $\mathcal L^{t,\C}_{V,\phi}K^{t,\bar\zeta}\bot \mathcal H (X_t, E_t)$. If both \eqref{eq:main1} and \eqref{eq:main2} are satisfied, then we have the variation formula
\begin{eqnarray*}
% \nonumber to remove numbering (before each equation)
  K_{t\bar t}^t(\zeta,\bar\eta) &=& \langle\langle K_{\bar t}^{t,\bar\eta}, K_{\bar t}^{t,\bar\zeta}\rangle\rangle_t
  +\langle\langle K^{t,\bar\eta}, \big[\mathcal L^t_{V,\phi}, \mathcal L^t_{\overline{V}}\big]K^{t,\bar\zeta} \rangle\rangle_t\\
   & & \ \ \ \ \ \ \ \ \ \ \ \ \ \ \ \ \ \ \  - \ i^{n^2}\int_{X_t}\{\delta_{\op^t V}K^{t,\bar\eta}, \delta_{\op^t V}K^{t,\bar\zeta}\}- \langle\langle \mathcal L^{t,\C}_{V,\phi}K^{t,\bar\eta}, \mathcal L^{t,\C}_{V,\phi}K^{t,\bar\zeta} \rangle\rangle_t.
\end{eqnarray*}
\end{theorem}

If $p$ is proper, then every smooth $(1,0)$-vector field is $p$-admissible. Thus for compact case, it suffices to assume that $\dim_{\C} H^0(X_t,E_t)$ is a constant.

Our main theorem implies that $\mathcal L^{t,\C}_{V,\phi}K^{t,\bar\eta}$ is the minimal solution of the following equation
\begin{equation}\label{eq:d-barLV}
    -\op^t \left(\mathcal L^{t,\C}_{V,\phi}K^{t,\bar\eta}\right)=\partial^t_{\phi}\left(\delta_{\op^t V}K^{t,\bar\eta}\right)+\left(\op^t \phi\right)_V\wedge K^{t,\bar\eta},
\end{equation}
where $\partial^t_{\phi}(\cdot):=e^{\phi}\partial^t(e^{-\phi}\cdot)$ and $\left(\op^t \phi\right)_V:=V\left(\op^t\phi\right)$. If $\Theta_{L_t}:=i\partial^t\op^t\phi^t>0$, then
\begin{equation}\label{eq:har-vec}
    \left(\op^t \phi\right)_V=0 \Leftrightarrow V=\frac{\partial}{\partial t}-\sum\phi_{t\bar k}\phi^{\bar k j}\frac{\partial}{\partial \mu^j},  \  \phi_{t\bar k}:=\frac{\partial^2\phi}{\partial t\partial \bar\mu^k}, \ (\phi^{\bar k j})=(\phi_{j \bar k})^{-1}.
\end{equation}
Put
\begin{equation}\label{eq:vc}
    V_{\phi}=\frac{\partial}{\partial t}-\sum\phi_{t\bar k}\phi^{\bar k j}\frac{\partial}{\partial \mu^j}, \
    c(\phi)=\phi_{t\bar t}-\sum \phi_{t\bar k} \phi_{\bar t j}\phi^{\bar k j}.
\end{equation}
We have $\delta_{\op^t V_{\phi}}\Theta_{L_t}=0$ and $\delta_{\op^t V_{\phi}}K^{t,\bar\eta}$ is $\Theta_{L_t}$-primitive (see \cite{Bern11} and Lemma 4 in \cite{Sch12}).

By Theorem~\ref{th:Main}, we shall give another proof of the following corollary without using \eqref{eq:secondorder}) and the curvature formula in \cite{Bern11}.

\begin{corollary}\label{co:compactphi}  If $p$ is proper and $\Theta_{L_t}>0, \ \forall \ t\in\mathbb D$, then we have
\begin{eqnarray*}
% \nonumber to remove numbering (before each equation)
   K_{t\bar t}^t(\zeta,\bar\eta) &=& \langle\langle K_{\bar t}^{t,\bar\eta}, K_{\bar t}^{t,\bar\zeta}\rangle\rangle_t
  +\langle\langle c(\phi)K^{t,\bar\eta}, K^{t,\bar\zeta}\rangle\rangle_t \\
   & &  \ \ \ \ \ \ \ \ \ \ \ \ \ \ \ \ \ \ \ \ \ \ \ \ \ + \ \langle\langle \left(\square'+1\right)^{-1}\delta_{\op^t V_{\phi}}K^{t,\bar\eta}, \delta_{\op^t V_{\phi}}K^{t,\bar\zeta}\rangle\rangle_{\Theta_{L_t}},
\end{eqnarray*}
where $\square'$ is the $\partial^t_{\phi}$-Laplace with respect to $\Theta_{L_t}$ and $\phi^t$.
\end{corollary}

Assume that $\mathcal X$ possesses a K\"ahler metric $\omega$. By Lemma 4.1 in \cite{Bern11}, there exist only one $V_{\omega}$ such that $p_*V_{\omega}=\partial/\partial t$ and $\delta_{V_\omega}\omega=c(\omega)d\bar t$, where $c(\omega)$ satisfies
\begin{equation}\label{eq:comega}
    \frac{\omega^{n+1}}{(n+1)!}=c(\omega)\frac{\omega^{n}}{n!}\wedge idt\wedge d\bar t.
\end{equation}
Let $T_{\omega_t}^{t,\bar \eta}$ be the harmonic part of $\delta_{\op^t V_{\omega}}K^{t,\bar\eta}$ with respect to K\"ahler metric $\omega_t:=\omega|_{X_t}$. The following corollary is due to Griffiths \cite{Griffiths84} (Berndtsson gave a new proof in \cite{Bern11}), but we shall also discuss how it follows from our main theorem.

\begin{corollary}\label{co:compact}  Assume that $\mathcal X$ is a K\"ahler manifold. If $p$ is proper and $\mathcal L$ is trivial, then we have
\begin{equation*}
 K_{t\bar t}^t(\zeta,\bar\eta) = \langle\langle K_{\bar t}^{t,\bar\eta}, K_{\bar t}^{t,\bar\zeta}\rangle\rangle_t
 + \langle\langle T_{\omega_t}^{t,\bar \eta}, T_{\omega_t}^{t,\bar \zeta}\rangle\rangle_{\omega_t}.
\end{equation*}
\end{corollary}

If $p$ is not proper, then not every $V$ is $p$-admissible. In fact, $V$ is $p$-admissible if and only if $V$ keeps the boundary. It is also well-known that stability of the Bergman kernels follows from regularity properties of the $\op$-Neumann problem (N) (see Lemma 2.1 in \cite{Bern06} and also \cite{Kohn}, \cite{Tartakoff80}, \cite{Kerzman72}, \cite{GreeneK82}, \cite{GreeneK81}, \cite{CCS94} for further results). For simple methods that rely only on H\"ormander's theory, see \cite{DiederichOhsawa91}. We shall study variation of planar domains, since (N) is an elliptic boundary problem for planar domain (see Theorem 10.5.3 in \cite{Hormander63} for stability properties of elliptic boundary problems).

Assume that $\mathcal X$ is a smoothly bounded domain in $\mathbb D\times\C$ and $p$ is the restriction of the canonical projection $\mathbb D\times\C\rightarrow\mathbb D$. Let $\rho$ be a smooth defining function of $\mathcal X$. Let $V$ (resp. $\phi$) be a smooth $(1,0)$-vector field (resp. function) on a neighborhood of the closure of $\mathcal X$ such that $p_*V=\partial/\partial t$. Then $V$ is $p$-admissible on $\mathcal X$ if and only if $V(\rho)=0$ on the boundary of $\mathcal X$ (consider the relative topology with respect to $\mathbb D\times\C$). If $V$ is $p$-admissible on $\mathcal X$, put
\begin{equation}\label{eq:k2}
k_2=\frac{\langle V, V\rangle_{i\partial\op\rho}}{|\rho_\mu|},
\end{equation}
then the value of $k_2$ on the boundary does not depend on $V$ and $\rho$. In fact, $k_2$ is the second boundary invariant (see (7) in \cite{MY04}) defined by Maitani and Yamaguchi. Let
\begin{equation}\label{eq:orthogonal}
    \delta_{\op^t V}K^{t,\bar\eta}=T^{t, \bar \eta}+S^{t, \bar \eta}, \ T^{t, \bar \eta}\in \ker\partial^t_{\phi},
    \ S^{t, \bar \eta}\ \bot\ \ker\partial^t_{\phi}.
\end{equation}
By \eqref{eq:d-barLV} and our main theorem, we shall prove that

\begin{corollary}\label{co:onephi}  Assume that $V$ is $p$-admissible. If $i\partial^t\op^t\phi^t>0, \ \forall \ t\in\mathbb D$, then we have
\begin{equation}\label{eq:onephi1}
    K_{t\bar t}^t(\eta,\bar\eta) \geq \int_{\partial X_t} k_2 |K^{t,\bar\eta}|^2 d \sigma
   + || K_{\bar t}^{t,\bar\eta}||^2_t  +  \langle\langle c(\phi)K^{t,\bar\eta}, K^{t,\bar\eta}\rangle\rangle_t+||T^{t, \bar\eta}||^2_t,
\end{equation}
where $d\sigma$ is the arc length element. If $\phi\equiv0$, then we have
\begin{equation}\label{eq:onephi2}
 K_{t\bar t}^t(\zeta,\bar\eta) = \int_{\partial X_t} k_2 \langle K^{t,\bar\eta}, K^{t,\bar\zeta}\rangle d \sigma
 +\langle\langle K_{\bar t}^{t,\bar\eta}, K_{\bar t}^{t,\bar\zeta}\rangle\rangle_t + \langle\langle T^{t,\bar \eta}, T^{t,\bar \zeta}\rangle\rangle_t.
\end{equation}
\end{corollary}

\eqref{eq:onephi2} is another form of formula (9) in \cite{MY04}. If $k_2\equiv0$, then $\partial\mathcal X$ is Levi-flat (foliated by holomorphic curves). The boundary of the graph of a holomorphic motion of a planar domain is Levi-flat. We shall use Corollary~\ref{co:onephi} to study holomorphic motions.

\begin{definition}[cf. \cite{MSS83}]\label{de:holo-motion} Let $X$ be a planar domain. A map $f: \mathbb D\times X \rightarrow \C$ is called a holomorphic motion of $X$ if
\begin{enumerate}[\upshape (i)]
  \item $f(0,z)\equiv z$ for all $z\in X$,
  \item For every $z\in X$, $f(\cdot,z)$ is holomorphic on $\mathbb D$,
  \item for every $t\in \mathbb D$, $f(t,\cdot)$ is injective on $X$.
\end{enumerate}
\end{definition}

Put $F(t,z)=(t,f(t,z))$. We call $\mathcal X:=F(\mathbb D\times X)$ the graph of $f$. We say that $f$ is trivial if the graph of $f$ is equal to the graph of a holomorphic motion $g$ of $X$ such that $g$ is holomorphic on $\mathbb D\times X$.

Put $V_f=F_*(\partial/\partial t)$, then $V_f$ is $p$-admissible on $\mathcal X$, where $p$ is the restriction of the canonical projection $\mathbb D\times X\rightarrow\mathbb D$. What's more, $V_f$ is integrable. As an application of \eqref{eq:onephi2}, we shall prove that

\begin{theorem}\label{th:Cflat} Let $X$ be a smoothly bounded planar domain. Let $f$ be a holomorphic motion of $X$. If $f$ is smooth up to the boundary, then the followings are equivalent:
\begin{enumerate}[\upshape (i)]
  \item $f$ is trivial,
  \item $K^t_{t\bar t}(\eta,\bar\eta)=||K_{\bar t}^{t,\bar\eta}||^2_{t}, \ \forall \ (t,\eta)\in F(\mathbb D\times \Omega)$,
  \item $\delta_{\op^t V_f}K^{t,\bar\eta}\perp \ker\partial^t, \ \forall \ (t,\eta)\in F(\mathbb D\times \Omega)$.
\end{enumerate}
\end{theorem}

Put $J=f_{\bar z}/f_z$, then (iii) is equivalent to
\begin{equation}\label{eq:trivialcondition}
    \int_{X_t} K^t(\zeta,\bar\eta) \left(\frac{(f_z)^2J_t}{|f_z|^2-|f_{\bar z}|^2}\right)(t,\zeta)\ id\zeta\wedge d\bar\zeta=0, \ \forall \ (t,\eta)\in F(\mathbb D\times X).
\end{equation}
In the last section, we shall use \eqref{eq:trivialcondition} to study affine holomorphic motions. We shall prove the following corollary of Theorem~\ref{th:Cflat}.

\begin{corollary}\label{co:last} Let $f=z+a(t)\bar z$ be a holomorphic motion of a smoothly bounded planar domain. $f$ is trivial if and only if $a\equiv0$ on $\mathbb D$.
\end{corollary}

In \cite{Liurenshan}, Ren-Shan Liu showed that if $f=z+t^2\bar z$, then $F(\mathbb D\times \mathbb D)$ is not biholomorphic equivalent to the bidisc. Corollary~\ref{co:last} is interesting, since every holomorphic motion of a subset of $\C$ can be extended to the whole complex plane (see \cite{Sl91} and \cite{ST86}).

Let's come back to the weighted case. Let $X$ be a smoothly bounded planar domain. Let $\phi$ be a smooth maximal plurisubharmonic function on a neighborhood of the closure of $\{0\}\times X$. Assume that $\phi(0,\cdot)$ is strictly subharmonic on a neighborhood of the closure of $X$. Since $(\partial\op\phi)^2=0$, $V_\phi$ is integrable. Denote by $f_{\phi}$ the holomorphic motion of $X$ induced by $V_\phi$. We shall prove the following non-compact version of Corollary~\ref{co:compactphi}.

\begin{theorem}\label{th:lastphi} With the notation above, the following variation formula follows
\begin{equation*}
    K_{t\bar t}^t(\zeta,\bar\eta) = \langle\langle K_{\bar t}^{t,\bar\eta}, K_{\bar t}^{t,\bar\zeta}\rangle\rangle_t
  +\langle\langle \left(\square'+1\right)^{-1}\delta_{\op^t V_{\phi}}K^{t,\bar\eta}, \delta_{\op^t V_{\phi}}K^{t,\bar\zeta}\rangle\rangle_t,
\end{equation*}
where $\square'$ is the $\partial^t_{\phi}$-Laplace with respect to $i\partial^t\op^t\phi^t$ and $\phi^t$.
\end{theorem}

\section{Variation of the Bergman kernels}

By \eqref{eq:K2}, variation of $K^t$ is connected with variation of fibre integrals. By Lemma 1 in \cite{Sch12}, one may use Lie derivatives to compute variation of fibre integrals. Note that
\begin{equation}\label{eq:3Lie}
    \mathcal L^t_V\{u^t,v^t\}=\{\mathcal L^t_{V,\phi}u^t,v^t\}+\{u^t,\mathcal L^t_{\overline V}v^t\},
\end{equation}
where $u^t, v^t\in C^{\infty}_{\bullet,\bullet}(X_t,L_t)$ whose admissible fibre coordinates depend smoothly on $t$.
We shall use \eqref{eq:3Lie} to prove Theorem~\ref{th:Main}.

\begin{proof}[Proof of Theorem~\ref{th:Main}.] By \eqref{eq:K2} and \eqref{eq:main1},
\begin{equation*}
    K^t_t(\zeta,\bar\eta)=i^{n^2}\int_{X_t}\mathcal L^t_V\{K^{t,\bar\eta},K^{t,\bar\zeta}\}.
\end{equation*}
By \eqref{eq:3Lie},
\begin{equation*}
    \mathcal L^t_V\{K^{t,\bar\eta},K^{t,\bar\zeta}\}=
    \{\mathcal L^t_{V,\phi}K^{t,\bar\eta},K^{t,\bar\zeta}\}+\{K^{t,\bar\eta},K_{\bar t}^{t,\bar\zeta}\}.
\end{equation*}
Thus if $||\mathcal L^t_{V,\phi}K^{t,\bar\zeta}||_t<\infty, \ ||K_{\bar t}^{t,\bar\zeta}||_t<\infty, \ \forall \ (t,\zeta)\in \mathcal X$, then
\begin{equation*}
     K^t_t(\zeta,\bar\eta)=i^{n^2}\int_{X_t}
     \{\mathcal L^t_{V,\phi}K^{t,\bar\eta},K^{t,\bar\zeta}\}
     +\langle\langle K^{t,\bar\eta},K_{\bar t}^{t,\bar\zeta}\rangle\rangle_t.
\end{equation*}
By \eqref{eq:firstorder} and \eqref{eq:relation1}, (note that $\delta_{\op^t V}K^{t,\bar\eta}$ is an $(n-1,1)$-form) we have
\begin{equation}\label{eq:chuizhi}
    i^{n^2}\int_{X_t}
     \{\mathcal L^t_{V,\phi}K^{t,\bar\eta},K^{t,\bar\zeta}\}=\langle\langle
     \mathcal L^{t,\C}_{V,\phi}K^{t,\bar\eta},K^{t,\bar\zeta}\rangle\rangle_t=0,  \ \forall \ (t,\zeta)\in\mathcal X,
\end{equation}
thus $\mathcal L^{t,\C}_{V,\phi}K^{t,\bar\zeta}\bot \mathcal H (X_t, E_t)$.

By \eqref{eq:main2} and \eqref{eq:firstorder}, we have
\begin{equation*}
    K^t_{t\bar t}(\zeta,\bar\eta)=\langle\langle K_{\bar t}^{t,\bar\eta},K_{\bar t}^{t,\bar\zeta}\rangle\rangle_t
    +i^{n^2}\int_{X_t}\{K^{t,\bar\eta}, \mathcal L^t_{V,\phi}K_{\bar t}^{t,\bar\zeta} \}.
\end{equation*}

Since $\int_{X_t}\{K^{t,\bar\eta},\mathcal L^t_{V,\phi}K^{t,\bar\zeta}\}=0$, by \eqref{eq:main2}, we have
\begin{equation*}
    \int_{X_t}\{\mathcal L^t_{V,\phi}K^{t,\bar\eta},\mathcal L^t_{V,\phi}K^{t,\bar\zeta}\}+
    \int_{X_t}\{K^{t,\bar\eta},\mathcal L^t_{\overline{V}}\mathcal L^t_{V,\phi}K^{t,\bar\zeta}\}=0,
\end{equation*}
thus
\begin{equation*}
    K^t_{t\bar t}(\zeta,\bar\eta)=\langle\langle K_{\bar t}^{t,\bar\eta},K_{\bar t}^{t,\bar\zeta}\rangle\rangle_t
    +\langle\langle  K^{t,\bar\eta},\big[\mathcal L^t_{V,\phi}, \mathcal L^t_{\overline{V}}\big]K^{t,\bar\zeta} \rangle\rangle_t
    -i^{n^2}\int_{X_t}\{\mathcal L^t_{V,\phi}K^{t,\bar\eta}, \mathcal L^t_{V,\phi}K^{t,\bar\zeta} \}.
\end{equation*}
By \eqref{eq:relation1}, our final formula follows.
\end{proof}

\begin{lemma}\label{le:lee} If $\Theta_{L_t}>0$, then
\begin{equation}\label{eq:leeformula}
   \left(\big[\mathcal L^t_{V,\phi}, \mathcal L^t_{\overline{V}}\big]K^{t,\bar\zeta}\right)_{(n,0)}-\big|(\op^t\phi)_V\big|^2_{\Theta_{L_t}}K^{t,\bar\zeta}
    =c(\phi)K^{t,\bar\zeta}+\partial^t_{\phi}\delta_{[V,\overline V]}K^{t,\bar\zeta},
\end{equation}
for every $p$-admissible smooth $(1,0)$-vector field $V$ on $\mathcal X$, where $(\cdot)_{(n,0)}$ means the $(n,0)$-component of a differential form.
\end{lemma}

\begin{proof} Let $V=\partial/\partial t-\sum \alpha^j\partial/\partial \mu^j$. Since
\begin{equation*}
    \big[\mathcal L^t_{V,\phi}, \mathcal L^t_{\overline{V}}\big]=\overline V V\phi+\mathcal L^t_{[V,\overline V]}, \ \left(\mathcal L^t_{[V,\overline V]} K^{t,\bar\zeta}\right)_{(n,0)}=\partial^t\delta_{[V,\overline V]}K^{t,\bar\zeta}
\end{equation*}
and
\begin{equation*}
    \overline V V\phi-\big|(\op^t\phi)_V\big|^2_{\Theta_{L_t}}=c(\phi)+\sum \left(\overline{\alpha^k}\alpha_{\bar k}^j-\alpha_{\bar t}^j\right)\phi_j=c(\phi)-\delta_{[V,\overline V]}\partial^t \phi,
\end{equation*}
the lemma follows.
\end{proof}

By the above lemma,
\begin{equation}\label{eq:withoutboundary}
    \langle\langle  K^{t,\bar\eta},\big[\mathcal L^t_{V,\phi}, \mathcal L^t_{\overline{V}}\big]K^{t,\bar\zeta} \rangle\rangle_t=\langle\langle  c(\phi)K^{t,\bar\eta},K^{t,\bar\zeta} \rangle\rangle_t,
\end{equation}
if $X_t$ is compact without boundary and $(\partial^t\phi)_V=0$. By \eqref{eq:har-vec}, it suffices to take $V=V_{\phi}$.

\begin{proof}[Proof of Corollary~\ref{co:compactphi}] Since $\delta_{\op^t V_{\phi}}K^{t,\bar\eta}$ is $\Theta_{L_t}$-primitive. It is sufficient to show that
\begin{equation}\label{eq:co151}
    \langle\langle \mathcal L^{t,\C}_{V_{\phi},\phi}K^{t,\bar\eta}, \mathcal L^{t,\C}_{V_{\phi},\phi}K^{t,\bar\zeta} \rangle\rangle_t=\langle\langle \square'\left(\square'+1\right)^{-1}\delta_{\op^t V_{\phi}}K^{t,\bar\eta}, \delta_{\op^t V_{\phi}}K^{t,\bar\zeta}\rangle\rangle_{\Theta_{L_t}}.
\end{equation}
Since $\mathcal L^{t,\C}_{V_{\phi},\phi}K^{t,\bar\eta}\bot\mathcal H(X_t,E_t)$, by \eqref{eq:d-barLV}, we have
\begin{equation*}
    \mathcal L^{t,\C}_{V_{\phi},\phi}K^{t,\bar\eta}=-(\op^t)^*(\square'')^{-1}\partial^t_{\phi}\left(\delta_{\op^t V_{\phi}}K^{t,\bar\eta}\right),
\end{equation*}
thus
\begin{equation*}
    \langle\langle \mathcal L^{t,\C}_{V_{\phi},\phi}K^{t,\bar\eta}, \mathcal L^{t,\C}_{V_{\phi},\phi}K^{t,\bar\zeta} \rangle\rangle_t=\langle\langle \left(\square''\right)^{-1}\partial^t_{\phi}\delta_{\op^t V_{\phi}}K^{t,\bar\eta}, \partial^t_{\phi}\delta_{\op^t V_{\phi}}K^{t,\bar\zeta}\rangle\rangle_{\Theta_{L_t}}.
\end{equation*}
Since $\square''=\square'+1$ and $(\square'+1)^{-1}$ commutes with $\partial^t_{\phi}$, it suffices to show that
\begin{equation*}
    (\partial^t_{\phi})^*\delta_{\op^t V_{\phi}}K^{t,\bar\zeta}=-*\op^t*\delta_{\op^t V_{\phi}}K^{t,\bar\zeta}
    =(-i)^{n^2}*\op^t\left(\delta_{\op^t V_{\phi}}K^{t,\bar\zeta}\right)=0.
\end{equation*}
Since $\delta_{\op^t V_{\phi}}K^{t,\bar\zeta}$ is $\op^t$-closed, Corollary~\ref{co:compactphi} follows.
\end{proof}

\begin{proof}[Proof of Corollary~\ref{co:compact}] Since $\omega_t$ is a K\"ahler metric on $X_t$, We have $\square''=\square'$ and
\begin{equation*}
    \mathcal L^{t,\C}_{V_{\omega}}K^{t,\bar\eta}=-(\op^t)^*G\partial^t\left(\delta_{\op^t V_{\omega}}K^{t,\bar\eta}\right),
\end{equation*}
where $G$ is the Green operator of $\square'$. Thus
\begin{equation*}
    \langle\langle \mathcal L^{t,\C}_{V_{\omega}}K^{t,\bar\eta}, \mathcal L^{t,\C}_{V_{\omega}}K^{t,\bar\zeta} \rangle\rangle_t=\langle\langle \delta_{\op^t V_{\omega}}K^{t,\bar\eta}-T_{\omega_t}^{t,\bar \eta}, \delta_{\op^t V_{\omega}}K^{t,\bar\zeta}-T_{\omega_t}^{t,\bar \zeta}\rangle\rangle_{\omega_t}.
\end{equation*}
Since $\omega$ is a K\"ahler metric, $\delta_{\op^t V_{\omega}}K^{t,\bar\eta}$ is $\omega_t$-primitive, hence Corollary~\ref{co:compact} follows from our main theorem.
\end{proof}

If $\omega$ is not a K\"ahler metric, $\delta_{\op^t V_{\omega}}K^{t,\bar\eta}$ may not be $\omega_t$-primitive. Thus the global K\"ahler assumption is necessary. By Kodaira-Spencer's theorem, if one fibre of $p$ possesses a K\"ahler metric, so does the nearby fibre and what's more, one may choose K\"ahler metrics $\{\omega_t\}_{t\in \mathbb D}$ such that $\omega_t$ depends smoothly on $t$. Zhi-Qin Lu told me that one may get a global Hermitian metric whose restriction to each fibre $X_t$ is $\omega_t$ by using partition of unity. Assume that $X_0$ possesses a K\"ahler metric, we still don't know whether $p^{-1}(U)$ possesses a K\"ahler metric or not, where $U$ is a sufficiently small neighborhood of $0$. We only know that the answer is negative when $U$ is sufficiently large. For example, the Iwasawa manifold $M_I$ possesses a Boothby-Wang fibration, in fact, $M_I$ can be seen as a torus bundle over a two dimensional torus (see section 7 in \cite{Foreman00}), but $M_I$ is not a K\"ahler manifold.

Let's discuss the non-compact case. Assume that $\mathcal X$ is a smoothly bounded domain in $\mathbb D\times\C^n$ and $p$ is the restriction of the canonical projection $\mathbb D\times\C^n\rightarrow\mathbb D$. Let $\rho$ be a defining function of $\mathcal X$. Assume that $V(\rho)=0$ on the boundary. We shall prove that

\begin{lemma}\label{le:boundary} If $\phi$ is smooth up to the boundary of $\mathcal X$, then
\begin{equation}\label{eq:boundary}
   \langle\langle  K^{t,\bar\eta},\partial^t_{\phi}\delta_{[V,\overline V]}K^{t,\bar\zeta} \rangle\rangle_t
   =\int_{\partial X_t}k_2\langle K^{t,\bar\eta},K^{t,\bar\zeta}\rangle d\sigma.
\end{equation}
\end{lemma}

\begin{proof} Put $d^t=d|_{X_t}$, we have
\begin{equation*}
    \{K^{t,\bar\eta},\partial^t_{\phi}\delta_{[V,\overline V]}K^{t,\bar\zeta}\}=(-1)^n d^t\{K^{t,\bar\eta},\delta_{[V,\overline V]}K^{t,\bar\zeta}\}.
\end{equation*}
Let $V=\partial/\partial t-\sum \alpha^j\partial/\partial \mu^j$, we have
\begin{equation*}
    [V,\overline V]=\sum \left( \overline V \alpha^j\right)\partial/\partial \mu^j-\left( V \overline{\alpha^j}\right)\partial/\partial \bar\mu^j.
\end{equation*}
On the boundary of $\mathcal X$, we have
\begin{equation*}
    d\sigma=\frac{\delta_{\rho_{\bar\mu}} i^{n^2}d\mu\wedge\overline{d\mu}}{|\rho_{\mu}|}=
    \frac{\delta_{\rho_{\mu}}i^{n^2}d\mu\wedge\overline{d\mu}}{|\rho_{\mu}|}
\end{equation*}
where
\begin{equation*}
    \rho_{\mu}=\sum\rho_k\partial/\partial\bar\mu^k, \ \rho_{\bar\mu}=\overline{\rho_{\mu}}, \ |\rho_\mu|=\sqrt{\sum|\rho_k|^2}.
\end{equation*}
Thus it suffices to show that
\begin{equation*}
    \sum\left( V \overline{\alpha^j}\right)\rho_{\bar j}=\langle V,V\rangle_{i\partial\op\rho},
\end{equation*}
on the boundary of $\mathcal X$.

By assumption, $V(\rho)=h\rho$, where $h$ is a smooth function near the boundary, thus
\begin{equation*}
    \sum\left( V \overline{\alpha^j}\right)\rho_{\bar j}=\langle V,V\rangle_{i\partial\op\rho}-(V\bar h+|h|^2)\rho.
\end{equation*}
The proof of Lemma~\ref{le:boundary} is complete.
\end{proof}

Now we can prove Corollary~\ref{co:onephi}. We shall use H\"ormander's $L^2$-estimates to prove \eqref{eq:onephi1} (see \cite{Hormander65}, \cite{Demailly12}, \cite{Demailly82}).

\begin{proof}[Proof of Corollary~\ref{co:onephi}] By the above two lemmas and our main theorem, if suffices to show that
\begin{equation}\label{eq:proofone}
    ||\mathcal L^{t,\C}_{V,\phi}K^{t,\bar\eta}||^2_t
    \leq ||S^{t,\bar\eta}||_t^2+
    \langle\langle |(\op^t\phi)_V|^2_{i\partial^t\op^t\phi^t}K^{t,\bar\eta},K^{t,\bar\eta}\rangle\rangle_t.
\end{equation}
By \eqref{eq:orthogonal} and \eqref{eq:d-barLV}, $-\mathcal L^{t,\C}_{V,\phi}K^{t,\bar\eta}$ is the $L^2$-minimal solution of
\begin{equation*}
    \op^t(\cdot)=u:=\partial^t_{\phi}\left(S^{t,\bar\eta}\right)+\left(\op^t \phi\right)_V\wedge K^{t,\bar\eta}.
\end{equation*}
Let $\widehat{\omega}$ be a complete K\"ahler metric on $X_t$. Let $f\in C_0^{\infty}(X_t, \wedge^{1,1}T^*X_t\otimes L_t)$ be a smooth form with compact support, we have
\begin{equation*}
    \langle\langle f,u\rangle\rangle_{\widehat{\omega}}=\big\langle\big\langle (\partial^t_{\phi})^*f,S^{t,\bar\eta}\big\rangle\big\rangle_{\widehat{\omega}}
    +\big\langle\big\langle f,\left(\op^t \phi\right)_V\wedge K^{t,\bar\eta}\big\rangle\big\rangle_{\widehat{\omega}},
\end{equation*}
If $f\in \ker \op^t$, then by the Bochner-Kodaira-Nakano equality,
\begin{equation*}
    ||(\op^t)^*f||_{\widehat{\omega}}^2-||(\partial_{\phi}^t)^*f||_{\widehat{\omega}}^2=\langle\langle [i\partial^t\op^t\phi^t,\widehat{L}]f, f\rangle\rangle_{\widehat{\omega}},
\end{equation*}
where $\widehat{L}$ is the adjoint of $\widehat{\omega}\wedge\cdot$. Thus
\begin{equation*}
    \big|\langle\langle f,u\rangle\rangle_{\widehat{\omega}}\big|^2\leq \left(||S^{t,\bar\eta}||^2_{\widehat{\omega}}+
    \big\langle\big\langle |(\op^t\phi)_V|^2_{i\partial^t\op^t\phi^t}K^{t,\bar\eta},K^{t,\bar\eta}\big\rangle\big\rangle_t\right) ||(\op^t)^*f||^2_{\widehat{\omega}}.
\end{equation*}
Since $||S^{t,\bar\eta}||_{\widehat{\omega}}=||S^{t,\bar\eta}||_t$, \eqref{eq:onephi1} follows from H\"ormander's theorem and the standard density technique on complete manifold.

If $\phi\equiv0$, we have
\begin{equation*}
    \mathcal L^{t,\C}_{V}K^{t,\bar\eta}=-(\op^t)^*(\square'')^{-1}\partial^t\left(\delta_{\op^t V}K^{t,\bar\eta}\right),
\end{equation*}
then
\begin{equation*}
    \langle\langle\mathcal L^{t,\C}_{V}K^{t,\bar\eta},\mathcal L^{t,\C}_{V}K^{t,\bar\zeta}\rangle\rangle_t
    =\langle\langle(\square'')^{-1}\partial^t\left(\delta_{\op^t V}K^{t,\bar\eta}\right), \partial^t\left(\delta_{\op^t V}K^{t,\bar\eta}\right)\rangle\rangle_{i\partial^t\op^t|\mu|^2}.
\end{equation*}
Write $S^{t,\bar\eta}=(\partial^t)^*\partial^tA^{t,\bar\eta}$. Since $\partial^tA^{t,\bar\eta}=0$ on the boundary, we have
\begin{equation*}
\square''\partial^tA^{t,\bar\eta}=\square'\partial^tA^{t,\bar\eta},
\end{equation*}
hence
\begin{equation*}
    \langle\langle\mathcal L^{t,\C}_{V}K^{t,\bar\eta},\mathcal L^{t,\C}_{V}K^{t,\bar\zeta}\rangle\rangle_t
    =\langle\langle\partial^tA^{t,\bar\eta}, \partial^t\left(\delta_{\op^t V}K^{t,\bar\eta}\right)\rangle\rangle_{i\partial^t\op^t|\mu|^2}
    =\langle\langle S^{t,\bar\eta},S^{t,\bar\zeta}\rangle\rangle_t.
\end{equation*}
By Lemma~\ref{le:boundary} and our main theorem, \eqref{eq:onephi2} follows.
\end{proof}

\section{Applications to holomorphic motion}

\begin{proof}[Proof of Theorem~\ref{th:Cflat}] Step 1. Proof of $(i)\Rightarrow(ii)$: Since $f$ is trivial, there exists a holomorphic motion $g$ with the same graph such that $g$ is holomorphic. Thus $V_g$ is a $p$-admissible holomorphic vector field on $\mathcal X$. By \eqref{eq:orthogonal}, $T^{t,\bar\eta}=0$, $(ii)$ follows.

Step 2. Proof of $(ii)\Rightarrow(iii)$: Since $k_2=0$ on the boundary, by \eqref{eq:onephi2} and $(ii)$, we have $T^{t,\bar\eta}=0$. Thus $(iii)$ follows.

Step 3. Proof of $(iii)\Rightarrow(i)$: Assume that $V_f=\partial/\partial t-\alpha\partial/\partial\zeta$. It suffices to find a holomorphic vector field $V$ on $\mathcal X$ such that $V=V_f$ on $\partial X_t, \ \forall \ t\in\mathbb D$. Since $d\bar\zeta\in \ker\partial^t$, by $(iii)$, we have
\begin{equation}\label{eq:step3}
    \int_{X_t}\alpha_{\bar\zeta}K^t(\zeta,\bar\eta)\ id\zeta\wedge d\bar\zeta=0,
\end{equation}
i.e. $\op^t\alpha\bot\ker\partial^t$. Hence there exits $\beta^t$ such that
\begin{equation}\label{eq:step31}
    \op^t\alpha=(\partial^t)^*(\beta^t id\zeta\wedge d\bar\zeta)=-i\op^t\beta^t.
\end{equation}
Let $\beta$ be a smooth function on a neighborhood of the closure of $\mathcal X$ such that
\begin{equation*}
    \beta|_{X_t}=\beta^t, \ \forall \ t\in\mathbb D.
\end{equation*}
Since $\beta^t id\zeta\wedge d\bar\zeta\in{\rm Dom}(\partial^t)^*$, we have
\begin{equation}\label{eq:step32}
    \beta|_{\partial X_t}=0, \ \forall \ t\in\mathbb D.
\end{equation}
Since $V_f$ is $p$-admissible, we have $\overline{V_f}\beta=0$ on the boundary of $\mathcal X$. Thus $\beta_{\bar t}-\bar \alpha\beta_{\bar \zeta}=0$ on the boundary. Since $\beta_{\bar \zeta}=i\alpha_{\bar\zeta}$ and $[V_f,\bar V_f]=0$, we have $\beta_{\bar t}-i\alpha_{\bar t}=0$ on the boundary. By \eqref{eq:step31}, $\beta_{\bar t}-i\alpha_{\bar t}$ is holomorphic on each fibre, thus
\begin{equation}\label{eq:step33}
    \beta_{\bar t}-i\alpha_{\bar t}\equiv0
\end{equation}
on $\mathcal X$. By \eqref{eq:step31}, \eqref{eq:step32} and \eqref{eq:step33},
\begin{equation*}
    V:=\partial/\partial t-(\alpha+i\beta)\partial/\partial\zeta
\end{equation*}
is a holomorphic vector field on $\mathcal X$ such that $V=V_f$ on $\partial X_t, \ \forall \ t\in\mathbb D$.
\end{proof}

Clearly, \eqref{eq:step3} is equivalent to $(iii)$. Since $V_f=F_*(\partial/\partial t)$, we have
$\alpha=-f_t$ on $\mathcal X$. Since
\begin{equation*}
    \overline{z_{\zeta}}=\frac{f_z}{|f_z|^2-|f_{\bar z}|^2}, \ z_{\bar\zeta}=\frac{-f_{\bar z}}{|f_z|^2-|f_{\bar z}|^2},
\end{equation*}
we have
\begin{equation}\label{eq:abz}
    \alpha_{\bar\zeta}=\frac{-(f_z)^2J_t}{|f_z|^2-|f_{\bar z}|^2}.
\end{equation}
Thus $\eqref{eq:trivialcondition}$ is equivalent to $(iii)$.

\begin{proof}[Proof of Corollary~\ref{co:last}] For affine holomorphic motion $f=z+a(t)\bar z$,
\begin{equation*}
    \alpha_{\bar\zeta}=\frac{-a'(t)}{1-|a(t)|^2}.
\end{equation*}
Thus $(iii)$ is equivalent to $a'(t)=0, \ \forall \ t\in\mathbb D$. Since $a(0)=0$, we get $(iii)$ is equivalent to $a\equiv0$. Thus Corollary~\ref{co:last} follows from Theorem~\ref{th:Cflat}.
\end{proof}

Let's prove Theorem~\ref{th:lastphi} finally.

\begin{proof}[Proof of Theorem~\ref{th:lastphi}] By \eqref{eq:d-barLV}, we have
\begin{equation*}
    \mathcal L^{t,\C}_{V_{\phi},\phi}K^{t,\bar\eta}=-(\op^t)^*(\square'')^{-1}\partial^t_{\phi}\left(\delta_{\op^t V_{\phi}}K^{t,\bar\eta}\right).
\end{equation*}
Thus
\begin{equation}\label{eq:last}
    \langle\langle\mathcal L^{t,\C}_{V_{\phi},\phi}K^{t,\bar\eta}, \mathcal L^{t,\C}_{V_{\phi},\phi}K^{t,\bar\zeta}\rangle\rangle_t=\langle\langle
    (\square'')^{-1}\partial^t_{\phi}\left(\delta_{\op^t V_{\phi}}K^{t,\bar\eta}\right),
    \partial^t_{\phi}\left(\delta_{\op^t V_{\phi}}K^{t,\bar\zeta}\right)\rangle\rangle_{i\partial^t\op^t\phi^t}.
\end{equation}
Write $S^{t,\bar\eta}=(\partial^t_{\phi})^*\partial^t_{\phi}A^{t,\bar\eta}$, we have $\partial^t_{\phi}\left(\delta_{\op^t V_{\phi}}K^{t,\bar\eta}\right)=\square'\partial^t_{\phi}A^{t,\bar\eta}$. Since $\partial^t_{\phi}A^{t,\bar\eta}\in{\rm Dom}(\partial^t_{\phi})^*$, we have $\partial^t_{\phi}A^{t,\bar\eta}=0$ on the boundary of $X_t$. Thus $\partial^t_{\phi}A^{t,\bar\eta}\in{\rm Dom}(\square'')$ and
\begin{equation}\label{eq:last1}
    \square'\partial^t_{\phi}A^{t,\bar\eta}=\square''\partial^t_{\phi}A^{t,\bar\eta}-\partial^t_{\phi}A^{t,\bar\eta}.
\end{equation}
By \eqref{eq:last}, we have
\begin{equation}\label{eq:last2}
    \langle\langle\mathcal L^{t,\C}_{V_{\phi},\phi}K^{t,\bar\eta}, \mathcal L^{t,\C}_{V_{\phi},\phi}K^{t,\bar\zeta}\rangle\rangle_t=\langle\langle
    \partial^t_{\phi}A^{t,\bar\eta}-(\square'+1)^{-1}\partial^t_{\phi}A^{t,\bar\eta},
    \partial^t_{\phi}\left(\delta_{\op^t V_{\phi}}K^{t,\bar\zeta}\right)\rangle\rangle_{i\partial^t\op^t\phi^t}.
\end{equation}
Since $\partial^t_{\phi}A^{t,\bar\eta}\in{\rm Dom}(\partial^t_{\phi})^*$, we have $(\square'+1)^{-1}\partial^t_{\phi}A^{t,\bar\eta}\in{\rm Dom}(\partial^t_{\phi})^*$ and
\begin{equation}\label{eq:last3}
    (\partial^t_{\phi})^*(\square'+1)^{-1}\partial^t_{\phi}A^{t,\bar\eta}=
    (\square'+1)^{-1}(\partial^t_{\phi})^*\partial^t_{\phi}A^{t,\bar\eta}=
    (\square'+1)^{-1}S^{t,\bar\eta}.
\end{equation}
By Lemma~\ref{le:boundary} and our main theorem, we have
\begin{equation}\label{eq:last4}
 K_{t\bar t}^t(\zeta,\bar\eta) =\langle\langle K_{\bar t}^{t,\bar\eta}, K_{\bar t}^{t,\bar\zeta}\rangle\rangle_t + \langle\langle T^{t,\bar \eta}, T^{t,\bar \zeta}\rangle\rangle_t +\langle\langle(\square'+1)^{-1}S^{t,\bar\eta},S^{t,\bar\zeta}\rangle\rangle_t.
\end{equation}
Since $\square'T^{t,\bar \eta}=0$, we have
\begin{equation}\label{eq:last5}
\langle\langle T^{t,\bar \eta}, T^{t,\bar \zeta}\rangle\rangle_t +\langle\langle(\square'+1)^{-1}S^{t,\bar\eta},S^{t,\bar\zeta}\rangle\rangle_t
=\langle\langle \left(\square'+1\right)^{-1}\delta_{\op^t V_{\phi}}K^{t,\bar\eta}, \delta_{\op^t V_{\phi}}K^{t,\bar\zeta}\rangle\rangle_t.
\end{equation}
The proof of Theorem~\ref{th:lastphi} is complete.
\end{proof}

\end{document}